\documentclass[11pt]{article}

\usepackage{amssymb}
\usepackage{amsmath}
\usepackage{amsthm}
\usepackage{amscd}
\usepackage{graphicx}
\usepackage{amsfonts}
\usepackage{xcolor}

\usepackage{hyperref}
\hypersetup{linktocpage}

\hypersetup{colorlinks,
    linkcolor={red!50!black},
    citecolor={blue!80!black},
    urlcolor={blue!80!black}}
\usepackage{float}

\newtheorem{thm}{Theorem}[section]
\newtheorem{cor}[thm]{Corollary}
\newtheorem{lem}[thm]{Lemma}

\newtheorem{prop}[thm]{Proposition}

\theoremstyle{definition}
\newtheorem{defn}[thm]{Definition}
\theoremstyle{remark}
\newtheorem{rem}[thm]{Remark}
\newtheorem{ex}[thm]{Example}

\newcommand{\id }{{\rm id}}
\renewcommand{\wr }{{\,\rm wr\,}}
\newcommand{\e }{\varepsilon }

\renewcommand{\phi }{\varphi}

\newcommand{\Orb}{\operatorname{Orb}}
\newcommand{\Sub}{\operatorname{Sub}}
\newcommand{\NN}{\mathbb{N}}
\newcommand{\ZZ}{\mathbb{Z}}
\newcommand{\G}{\mathcal{G}}

\begin{document}

\title{Finitely presented condensed groups}

\author{D. Osin}
\date{}
\maketitle

\begin{abstract}
Let $\G$ denote the space of finitely generated marked groups. For any finitely generated group $G$, we construct a continuous, injective map  $f$ from the space of subgroups $\Sub(G)$ to $\G$ that sends conjugate subgroups to isomorphic marked groups; in addition, if $G$ is finitely presented and $H\le G$ is finitely generated, then $f(H)$ is finitely presented. This result allows us to transfer various topological phenomena occurring in $\Sub(G)$ to $\G$. In particular, we provide the first example of a finitely presented group whose isomorphism class in $\G$ has no isolated points.
\end{abstract}

\section{Introduction}

Informally, the space of finitely generated marked groups, denoted by $\G$, is the set of all pairs $(G,A)$, where $G$ is a group and $A$ is a finite ordered generating set of $G$, considered up to a natural equivalence relation. A topology induced by the local convergence of Cayley graphs turns $\mathcal G$ into a $\sigma$-compact Polish space. We refer the reader to the next section for the precise definition.

We say that a marked group $(G,A)\in \G$ has a certain group-theoretic property (e.g., is finitely presented) if $G$ has this property. Similarly, two marked groups $(G,A),(H,B)\in \G$ are \emph{isomorphic} if $G\cong H$. For a finitely generated group $G$, we denote by $[G]$ its \emph{isomorphism class} in $\G$; that is,
$$
[G]=\{ (H,B)\in \G\mid H\cong G \}.
$$

The following definition was proposed in  \cite{Osi21a}.

\begin{defn}
A finitely generated group $G$ is \emph{condensed} if the isomorphism class $[G]$ has no isolated points in $\G$.
\end{defn}

The study of condensed groups is motivated by the connections to the Borel complexity of the isomorphism relation and model theoretic properties of subspaces of $\G$. For instance, it is not difficult to show that the isomorphism relation on a closed isomorphism-invariant subset $\mathcal S$ of $\G$ is smooth if and only if $\mathcal S$ contains no condensed groups. Furthermore, condensed groups lead to non-trivial examples of subspaces of $\G$ satisfying a topological zero-one law for first order sentences. In particular, the closure of the isomorphism class of any condensed group in $\G$ contains a subset of cardinality $2^{\aleph_0}$ consisting of finitely generated, elementarily equivalent, non-isomorphic groups. Finding examples of this kind is a challenging task since the standard tools for constructing elementarily equivalent models, such as ultrapowers and the L\" owenheim-Skolem theorem, are not available in the context of finitely generated structures. For more details and applications, we refer the interested reader to \cite{Osi21a}.

Most naturally occurring groups, such as linear groups, hyperbolic groups, and finitely presented residually finite groups, cannot be condensed. On the other hand, the class of condensed groups includes finitely generated groups isomorphic to their direct square, generic torsion-free lacunary hyperbolic groups, as well as certain solvable groups and groups of intermediate growth (see \cite{Nek,Osi21a,Osi21b,Wil}). All these examples are infinitely presented.

In \cite{Osi21a}, the author asked whether there exists a finitely presented condensed group. An additional motivation for this question comes from the fact that every such a group $G$ is \emph{extremely non-Hopfian} in the following sense: for every finite subset $\mathcal F\subseteq G$, there exists a non-injective epimorphism $\phi\colon G\to G$ such that $\phi\vert_{\mathcal F}$ is injective.

The goal of this note is to provide a general recipe for constructing finitely presented condensed groups. In order to state our main results, we need to introduce auxiliary notation. Let $\Sub(G)$ denote the space of subgroups of a group $G$. We think of $\Sub(G)$ as a subset of $2^G$ and endow it with the topology induced by the product topology on $2^G$. The group $G$ acts on $\Sub(G)$ continuously by conjugation: an element $g\in G$ maps each subgroup $H\le G$ to $gHg^{-1}$. We first prove the following theorem of independent interest.

\begin{thm}\label{main}
For any finitely generated group $G$, there is an injective, continuous map $f\colon Sub(G)\to \G$ satisfying the following conditions.
\begin{enumerate}
\item[(a)]$f $ maps conjugate subgroups of $G$ to isomorphic marked groups.
\item[(b)] If $G$ is finitely presented and $H\le G$ is finitely generated, then $f(H)$ is finitely presented.
\end{enumerate}
\end{thm}

Theorem \ref{main} allows us to transfer various topological phenomena occurring in $Sub(G)$ to the space $\G$. For instance, it implies that $f(H)$ is condensed whenever the set
$$
\Orb_G(H)=\{ gHg^{-1}\mid g\in G\}\subseteq \Sub (G)
$$
has no isolated points (see Corollary \ref{Cor:main}). Using Higman's embedding theorem, it is not difficult to show that a \emph{finitely generated group $H$ can be embedded in a finitely presented group $G$ so that $\Orb_G(H)$ has no isolated points if and only if $H$ is recursively presented and not co-Hopfian} (Proposition~\ref{prop}). This yields a variety of finitely presented condensed groups. In certain cases, e.g., for $H\cong \ZZ$, the use of Higman's embedding theorem can be avoided and we obtain particularly simple examples.

\begin{cor}\label{cor}
The group
$$
E=\left\langle a,b,c, h,s, t \; \left| \begin{array}{c}
[a,h]=1, \;\, [b,h]=1,\;\, [c,h]=1, \\
a^2=1, \;\, [a,a^b]=1, \;\,[b,c]=1,\;\, a^c=aa^b,\\
(h^2)^s= ha,\;\, (h^2)^t=h^2
\end{array}
\right. \right\rangle
$$
is condensed.
\end{cor}

\paragraph{Acknowledgments.} This work has been supported by the NSF grant DMS-1853989.

\section{The space of finitely generated marked groups}

We begin by reviewing the necessary definitions. Let $2^G$ denote the power set of $G$ endowed with the product topology (or, equivalently, the topology of pointwise convergence of indicator functions). For a group $G$, let $\Sub(G)$ denote the set of all subgroups of $G$. We think of $\Sub(G)$ as a subset of $2^G$ and endow it with the induced topology. Thus, the base of neighborhoods of $H\in \Sub(G)$ consists of the sets
\begin{equation}\label{U}
U(\mathcal F,H)=\{ K\le G \mid K\cap \mathcal F = H\cap \mathcal F\},
\end{equation}
where $\mathcal F$ ranges in the set of all finite subsets of $G$. $Sub(G)$ can be thought of as a particular case of the space of closed subgroups of a locally compact group introduced by Chabauty in \cite{Cha}; for this reason, $\Sub(G)$ is sometimes called the\emph{ Chabauty space} of $G$.

\begin{defn}[Grigorchuk \cite{Gri}]
Let $n\in \NN$. An \emph{$n$-generated marked group} is the equivalence class of a pair $(G,A)$, where $G$ is a group and $(a_1, \ldots, a_n)\subseteq G^n$ is an $n$-tuple such that $G$ is generated by $\{a_1, \ldots, a_n\}$; two such pairs  $(G,(a_1, \ldots, a_n))$ and $(H,(b_1, \ldots, b_n))$ are \emph{equivalent} if the map sending $a_i$ to $b_i$ for all $i=1, \ldots, n$ extends to an isomorphism $G\to H$. We keep the notation $(G,A)$ for the equivalence class of $(G,A)$.
\end{defn}

\vspace{.5mm}

\begin{ex}
Let $G=\ZZ^n$. For any $n\in \NN$ and any two generating tuples $A, B\in G^n$, the pairs $(G,A)$ and $(G,B)$ are equivalent.
\end{ex}

\vspace{.5mm}

Let $\G_n$ denote the set of all $n$-generated marked groups. The topology on $\G_n$ can be defined as follows. Let $F_n$ denote the free group of rank $n\ge 1$ with a fixed basis $X=\{x_1, \ldots, x_n\}$. Given a group $G$, a tuple $(a_1, \ldots, a_n)\in G^n$, and an element $w\in F_n$, we denote by $w(a_1, \ldots, a_n)$ the image of $w$ in $G$ under the ``evaluation homomorphism" sending $x_i$ to $a_i$ for all $i=1, \ldots, n$.
Given two marked groups $(G,A), (H,B)\in \G_n$, where $A=(a_1, \ldots, a_n)$ and $B=(b_1, \ldots, b_n)$, we write $(G,A)\cong_r(H,B)$ for some $r\in \NN$ if, for every element $w\in F_n$ of length $|w|_X\le r$, we have $w(a_1, \ldots, a_n)=1$ in $G$ if and only if $w(b_1, \ldots, b_n)=1$ in $H$. It is easy to see that the relations $\cong_r$ on $\G_n$ are well-defined, i.e., the definition is independent of the choice of particular representatives of the equivalence classes $(G,A)$ and $(H,B)$. The base of neighborhoods of a point $(G,A)\in \G_n$ is given by the sets
\begin{equation}\label{W}
W(r, (G,A))=\{ (H,B)\in \G_n \mid (H,Y)\cong_r (G,A)\}.
\end{equation}

\vspace{.5mm}

\begin{ex}
We have $(\ZZ/i\ZZ, (1))\cong_{i-1}(\ZZ, (1))$ and hence $\lim\limits_{i\to \infty}(\ZZ/i\ZZ, (1))=(\ZZ, (1))$ in $\G_1$.
\end{ex}

\vspace{.5mm}

\begin{prop}[Grigorchuk {\cite{Gri}}]
For every $n\in \NN$, $\G_n$ is a compact, separable, Hausdorff space. In particular, $\G_n$ is Polish.
\end{prop}

It is easy to see that the map $(G, (a_1, \ldots, a_n))\mapsto (G, (a_1, \ldots, a_n, 1))$ defines a continuous embedding $\G_n \to \G_{n+1}$. The topological union $$\G=\bigcup_{n=1}^\infty \G_n$$ is called the \emph{space of finitely generated marked groups}. The topology on $\G$ can be equivalently described as follows: a sequence $\{(G_i, A_i)\}_{i\in \NN}$ converges to $(G, A)$ in $\G$ if and only if there exist natural numbers  $n$ and $M\in \NN$ such that $(G,A)\in \G_n$, $(G_i, A_i)\in \G_n$ for all $i\ge M$, and the subsequence $\{(G_i, A_i)\}_{i=M}^\infty$ converges to $(G,A)$ in $\G_n$.

\section{Proof of Theorem \ref{main}}

For a subgroup $H$ of a group $G$, we denote by $E(G,H)$ the HNN-extension corresponding to the identical homomorphism $H\to H$. Assuming that $G$ has a presentation $\langle A\mid \mathcal R\rangle $ and thinking of each $h\in H$ as a word in the alphabet $A^{\pm 1}$, we have
\begin{equation}\label{EGH}
E(G,H)=\langle A, t\mid \mathcal R, \;\, t^{-1}ht=h\;\, \forall\, h\in H\rangle.
\end{equation}

\begin{rem}\label{Rem:FP}
If $G$ is finitely presented and $H$ is finitely generated, then $E(G,H)$ admits a finite presentation. Indeed, it suffices to impose the relations $t^{-1}ht=h$ for generators of $H$.
\end{rem}

To make our paper self-contained, we recall Britton's lemma on HNN-extensions. We state it in the particular case when the stable letter commutes with the associated subgroup.

\begin{lem}[Britton, \cite{Bri}]\label{BL}
Let $G$, $H$, and $E(G,H)$ be as above. Suppose that $$g_0t^{\e_1}g_1\ldots t^{\e_k}g_k=1$$ in $E(G,H)$, where
where $g_0, \ldots, g_k\in G$ and $\e_1, \ldots, \e_k\in \{\pm 1\}$. Then either $k=0$ and $g_0=1$ or there exists $j\in \{ 1, \ldots, k-1\}$ such that $\e_j=-\e_{j+1}$ and $g_j\in H$.
\end{lem}

\begin{lem}\label{Iso}
If $H_1$, $H_2$ are conjugate subgroups of a group $G$, then $E(G,H_1)\cong E(G, H_2)$.
\end{lem}

\begin{proof}
Let $G=\langle A\mid \mathcal R \rangle $. By definition, we have
\begin{equation}\label{EGHi}
E(G,H_i)=\langle A, t\mid \mathcal R, \;\, t^{-1}ht=h\;\, \forall\, h\in H_i\rangle
\end{equation}
for $i=1,2$. Suppose that $H_2=gH_1g^{-1}$ for some $g\in G$. Using (\ref{EGHi}), it is straightforward to verify that the map $a\mapsto a$ for all $a\in A$ and $t\mapsto g^{-1}tg$ extends to a homomorphism $\alpha\colon E(G, H_1)\to E(G, H_2)$. Similarly, the map $a\mapsto a$ for all $a\in A$ and $t\mapsto gtg^{-1}$ extends to a homomorphism $\beta\colon E(G, H_2)\to E(G, H_1)$. Since $\alpha\circ \beta$ and $\beta\circ \alpha$ are the identical maps on $E(G, H_2)$ and $E(G, H_1)$, respectively, we have $E(G,H_1)\cong E(G, H_2)$.
\end{proof}

We are now ready to prove the main result of our paper.

\begin{proof}[Proof of Theorem \ref{main}]
Let $G$ be a group generated by a set $A=\{a_1, \ldots, a_n\}$. We will show that the map $f\colon \Sub(G)\to \mathcal G_{n+1}$ defined by the formula $$f(H)= (E(G,H), (a_1, \ldots, a_n, t))\;\;\;\; \forall\, H\le G,$$ where $E(G,H)$ is given by (\ref{EGH}), satisfies all the requirements. Throughout the proof, $F_{n+1}$ denotes the free group with a basis $X=\{ x_1, \ldots, x_{n+1}\}$.

Let us first show that $f$ is injective. Suppose that $H_1\ne H_2$ are two subgroups of $G$. Without loss of generality, we can assume that there exists $h\in H_1\setminus H_2$. Then $[h, t_1]=1$ in $E(G, H_1)$ while $[h, t_2]\ne 1$ in $E(G, H_2)$ by Britton's Lemma. Therefore, the pairs $(E(G,H_1), (a_1, \ldots, a_n, t_1))$ and $(E(G,H_2), (a_1, \ldots, a_n, t_2))$ are not equivalent, i.e., they represent distinct elements of $\G$.

Further, we show that $f$ is continuous at every $H\in \Sub(G)$. To this end, for every natural number $r$, we need to find a finite subset $\mathcal F\subseteq G$ such that
\begin{equation}\label{cont}
f(U(\mathcal F, H))\subseteq W(r, f(H)),
\end{equation}
where the neighborhoods $U(\mathcal F, H)$ and $W(r, f(H))$ are defined by (\ref{U}) and (\ref{W}), respectively.

For an element $g\in G$, we denote by $|g|_A$ its length with respect to the generating set $A$. Let
$$
\mathcal F=\{ g\in G \mid |g|_A\le r\}.
$$
Proving the inclusion (\ref{cont}) amounts to showing that $f(K) \cong _r f(H)$ for any $K \in U(\mathcal F, H)$. Arguing by contradiction, suppose that there exists a word
$$
w=f_0x_{n+1}^{\e_1}f_1\ldots x_{n+1}^{\e_k}f_k,
$$
where each $f_i$ is a word in the alphabet $\{ x_1^{\pm 1}, \ldots, x_n^{\pm 1}\}$ and $\e_i=\pm 1$, such that $|w|_X\le r$ and

\medskip

($\ast$) \;\;\;\; $w(a_1, \ldots, a_n, t) =1$ in $E(G,H)$ but $w(a_1, \ldots, a_n, t) \ne 1$ in $E(G,K)$, or vice versa.

\medskip

\noindent Without loss of generality, we can assume that $w$ is the shortest element of $F_{n+1}$ satisfying these conditions.

In both groups $E(G,H)$ and $E(G, K)$, we have
$$
w(a_1, \ldots, a_n, t)=g_0t^{\e_1}g_1\ldots t^{\e_k}g_k,
$$
where $g_0=f_0(a_1,\ldots, a_n), \ldots, g_k=f_k(a_1,\ldots, a_n)$ are elements of $G$. Note that we necessarily have $k\ge 1$ since the natural maps from $G$ to $E(G,H)$ and $E(G, K)$ are injective. Applying Britton's lemma, we conclude that there must exist $j\in \{ 1, \ldots, k-1\}$ such that $\e_{j}=-\e_{j+1}$ and $g_{j}\in H$ or $g_j\in K$. Obviously, $$|g_j|_A\le |f_j|_X \le |w|_X \le r.$$ Hence $g_j\in \mathcal F$. Since $K \in U(\mathcal F, H)$,  $g_j$ must belong to both $H$ and $K$. Therefore, we have $t^{\e_j}g_jt^{\e_{j+1}}=g_j$ in both $E(G,H)$ and $E(G, K)$. This means that the word $u$ obtained from $w$ by replacing the subword $x_{n+1}^{\e_j}f_jx_{n+1}^{\e_{j+1}}$ with $f_j$, satisfies
$$
u(a_1,\ldots , a_n, t)= w(a_1, \ldots, a_n, t)
$$
in both $E(G,H)$ and $E(G, K)$. Clearly, $|u|_X\le |w|_X-2< |w|_X$, which contradicts the choice of $w$ as the shortest word of $F_{n+1}$ satisfying ($\ast$). This contradiction shows that $f(K) \in W(r, f(H))$ for all $K \in U(\mathcal F, H)$ and completes to proof of continuity. Lemma \ref{Iso} and Remark \ref{Rem:FP} imply properties (a) and (b).
\end{proof}

\begin{rem}
By the Higman-Neumann-Neumann theorem, every countable group $C$ embeds in a finitely generated group $G$. Furthermore, if $C$ is recursively presented, it embeds in a finitely presented group $G$ by the Higman embedding theorem. (Recall that a group is \emph{recursively presented} if it admits a presentation with a finite set of generators and recursively enumerable set of relations.) These embeddings induce  continuous maps $\Sub(C)\to \Sub(G)$ and allow us to generalize Theorem \ref{main} to all countable groups; in these settings, condition (b) will read as follows: \emph{if $G$ is recursively presented and $H$ is finitely generated, then $f(H)$ is finitely presented.} Of course, this generalization comes at the cost of making the map $f$ less explicit.
\end{rem}

We record an immediate corollary.

\begin{cor}\label{Cor:main}
Let $G$ be a finitely generated group. If $H\le G$ and $Orb_G(H)$ has no isolated points, then $E(G,H)$ is condensed.
\end{cor}

\begin{proof}
Let $f\colon \Sub(G)\to \G$ be the map provided by Theorem \ref{main}. We have $f(K)\in [E(G,H)]$ for all $K\in \Orb_G(H)$ by Theorem \ref{main} (a). Since $f$ is continuous and injective, the image of $\Orb_G(H)$ under $f$ has no isolated points. Hence, $[E(G,H)]$ is non-discrete. By \cite[Corollary 6.1]{Osi21a}, the isomorphism class of every finitely generated group in $\G$ is either discrete or has no isolated points. Therefore, $E(G,H)$ is condensed.
\end{proof}

\section{Constructing finitely presented condensed groups}\label{Sec2}

Corollary \ref{Cor:main} reduces the problem of constructing a finitely presented condensed group to finding an example of a finitely presented group $G$ and a finitely generated subgroup $H\le G$ such that $Orb_G(H)$ has no isolated points.  In this section, we construct examples of such pairs.

Given a group presentation $P=\langle X\mid \mathcal R\rangle $, by $Q=\langle P, Y\mid \mathcal S\rangle$ we denote the presentation obtained by adding a set $Y$ of new generators and a set $\mathcal S$ of new relations to $P$; thus, $Q=\langle X,Y\mid \mathcal R, \mathcal S\rangle$. We also employ the notation $x^y=y^{-1}xy$ for elements $x$, $y$ of a group. For example, in the settings of Theorem \ref{main}, we have $E(G,H)=\langle G, t \mid h^t=h \;\, \forall\, h\in H\rangle$.

We begin by considering a particular example. Let
$$
B=\langle a, b, c\mid a^2=1, [a,a^b]=1, [b,c]=1, a^c=aa^b \rangle .
$$
It is well-known and easy to check that the subgroup $\langle a, b\rangle $ of $B$ is isomorphic to $\ZZ/2\ZZ \wr \ZZ$ via the map sending $a$ (respectively, $b$) to a generator of $\ZZ/2\ZZ$ (respectively, $\ZZ$); for details and a more general embedding theorem for metabelian groups, see \cite{Bau73}. In particular, the elements $a^{b^i}$, $i\in \ZZ$, form a basis of a free abelian subgroup of exponent $2$ in $B$.

The elements $h^2$ and $ha$ of the direct product $\langle h\rangle\times B$ generate infinite cyclic subgroups. We denote by
$$
G= \langle B, h,s  \mid [h,a]=[h,b]=[h,c]=1,  (h^2)^s= ha\rangle
$$
the corresponding HNN-extension of $\langle h\rangle\times B$.  Let also $H=\langle h^2\rangle$. In this notation, we have the following.

\begin{lem}\label{OrbGH}
The subset $Orb_G(H)\subseteq Sub(G)$ has no isolated points.
\end{lem}

\begin{proof}
Since the action of $G$ on $\Sub(G)$ is continuous, it suffices to show that $H$ is a limit point of $Orb_G(H)$. Let $\mathcal F$ be a finite subset of $G$. We want to show that there exists $g\in G$ such that $gHg^{-1}\ne H$ and
\begin{equation}\label{HFH}
H\cap \mathcal F=gHg^{-1}\cap \mathcal F.
\end{equation}

By construction, the subgroup $A$ of $G$ generated by $\{ h\}\cup \{ a^{b^i}\mid i \in \ZZ\}$ is isomorphic to the direct product of $\ZZ=\langle h\rangle$ and $\bigoplus_{i\in \ZZ} A_i$, where $A_i=\langle a^{b^i}\rangle \cong \ZZ/2\ZZ$. In particular, $A$ is not finitely generated. Therefore, there exists $i\in \ZZ$ such that
\begin{equation}\label{abi}
a^{b^i}\notin \langle h, \mathcal F\cap A\rangle .
\end{equation}
We fix such an integer $i$ and let $g=(sb^i)^{-1}$. Using relations of $G$, we obtain
\begin{equation}\label{gHg}
gHg^{-1}=\langle (h^2)^{sb^i}\rangle =\langle (ha)^{b^i}\rangle = \langle ha^{b^i}\rangle.
\end{equation}
Since $(ha^{b^i})^2 = h^2$, we have $\langle ha^{b^i}\rangle = H\cup Ha^{b^i}$. By the choice of $i$, we have $Ha^{b^i}\cap \mathcal F=\emptyset$. Indeed, if there exists $f\in Ha^{b^i}\cap \mathcal F$, then $f\in A\cap \mathcal F$ and $a^{b^i}\in Hf$, which contradicts (\ref{abi}).
Therefore, $\langle ha^{b^i}\rangle \cap \mathcal F = H\cap \mathcal F$. Combining this with (\ref{gHg}) yields (\ref{HFH}).
\end{proof}

\begin{proof}[Proof of Corollary \ref{cor}]
The result follows immediately from Lemma \ref{OrbGH} and Corollary \ref{Cor:main}.
\end{proof}

\begin{rem}
As we mentioned in the introduction, every finitely presented condensed group is extremely non-Hopfian. This property of the group $E$ from Corollary \ref{cor} can be verified directly. Indeed, using the arguments similar to those in the proof of Theorem \ref{main} one can show that:
\begin{enumerate}
\item[(a)] for every $i\in \ZZ$, the map $a\mapsto a$, $b\mapsto b$, $c\mapsto c$, $h\mapsto h$, $s\mapsto s$, $t\mapsto t^{sb^i}$ extends to a surjective, non-injective homomorphism $\e_i\colon E\to E$;
\item[(b)] for any finite $\mathcal F\subseteq E$, all but finitely maps $\e_i$ are injective on $\mathcal F$.
\end{enumerate}
To show that the map defined in (a) extends to a homomorphism, one has to use the relation $sh^2s^{-1}=h^4$, which follows from $(h^2)^s=ha$, $[h,a]=1$, and $a^2=1$. The proof of (b) and non-injectivity of $\e_i$ employs Britton's lemma. We leave verifying details to the interested reader.
\end{rem}

Our approach to constructing condensed groups makes it natural to ask when a finitely generated subgroup $H$ can be embedded in a finitely presented group $G$ such that $Orb_H(G)$ has no isolated points. The proposition below answers this question and can be used to construct additional examples of finitely presented condensed groups. Since this proposition is unnecessary for obtaining the main results of our paper, we only give a sketch of the proof and leave details to the reader. Recall that a group $H$ is \emph{non-co-Hopfian} if there exists an injective, non-surjective homomorphism $H\to H$.

\begin{figure}
  \begin{center}
\begingroup%
  \makeatletter%
  \providecommand\color[2][]{%
    \errmessage{(Inkscape) Color is used for the text in Inkscape, but the package 'color.sty' is not loaded}%
    \renewcommand\color[2][]{}%
  }%
  \providecommand\transparent[1]{%
    \errmessage{(Inkscape) Transparency is used (non-zero) for the text in Inkscape, but the package 'transparent.sty' is not loaded}%
    \renewcommand\transparent[1]{}%
  }%
  \providecommand\rotatebox[2]{#2}%
  \newcommand*\fsize{\dimexpr\f@size pt\relax}%
  \newcommand*\lineheight[1]{\fontsize{\fsize}{#1\fsize}\selectfont}%
  \ifx\svgwidth\undefined%
    \setlength{\unitlength}{216.80544947bp}%
    \ifx\svgscale\undefined%
      \relax%
    \else%
      \setlength{\unitlength}{\unitlength * \real{\svgscale}}%
    \fi%
  \else%
    \setlength{\unitlength}{\svgwidth}%
  \fi%
  \global\let\svgwidth\undefined%
  \global\let\svgscale\undefined%
  \makeatother%
  \begin{picture}(1,0.5432)%
    \lineheight{1}%
    \setlength\tabcolsep{0pt}%
    \put(0,0){\includegraphics[width=\unitlength,page=1]{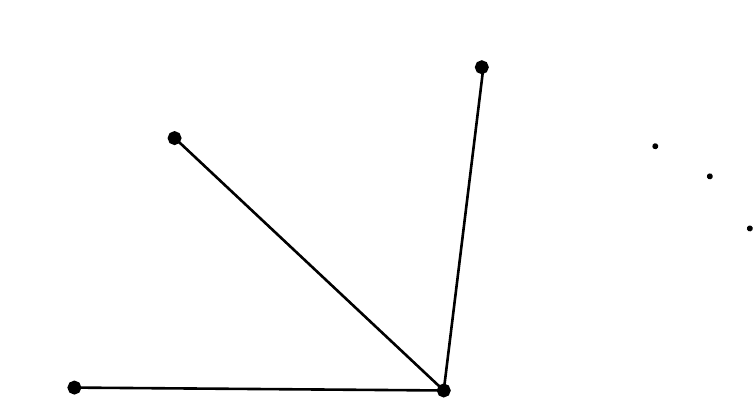}}%
    \put(0.6341536,0.01979933){\color[rgb]{0,0,0}\makebox(0,0)[lt]{\lineheight{1.25}\smash{\begin{tabular}[t]{l}$H$\end{tabular}}}}%
    \put(-0.00418902,0.00975637){\color[rgb]{0,0,0}\makebox(0,0)[lt]{\lineheight{1.25}\smash{\begin{tabular}[t]{l}$H_1$\end{tabular}}}}%
    \put(0.16096276,0.40207443){\color[rgb]{0,0,0}\makebox(0,0)[lt]{\lineheight{1.25}\smash{\begin{tabular}[t]{l}$H_2$\end{tabular}}}}%
    \put(0.62935656,0.50060709){\color[rgb]{0,0,0}\makebox(0,0)[lt]{\lineheight{1.25}\smash{\begin{tabular}[t]{l}$H_3$\end{tabular}}}}%
    \put(0.28641541,0.0549498){\color[rgb]{0,0,0}\makebox(0,0)[lt]{\lineheight{1.25}\smash{\begin{tabular}[t]{l}$H$\end{tabular}}}}%
    \put(0.40188175,0.22696408){\color[rgb]{0,0,0}\makebox(0,0)[lt]{\lineheight{1.25}\smash{\begin{tabular}[t]{l}$H$\end{tabular}}}}%
    \put(0.64168642,0.23697881){\color[rgb]{0,0,0}\makebox(0,0)[lt]{\lineheight{1.25}\smash{\begin{tabular}[t]{l}$H$\end{tabular}}}}%
  \end{picture}%
\endgroup%

  \caption{The star of groups in the proof of Proposition \ref{prop}.}\label{Fig1}
  \end{center}
\end{figure}
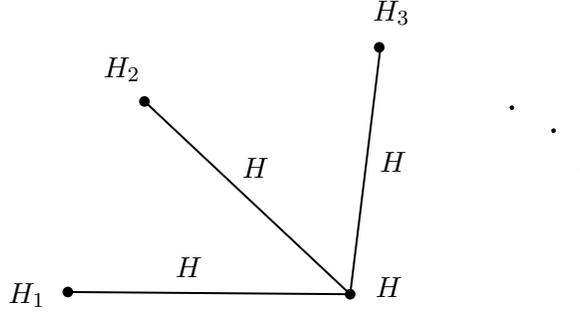

\begin{prop}\label{prop}
For any finitely generated group $H$, the following conditions are equivalent.
\begin{enumerate}
\item[(a)] There exists a finitely presented group $G$ containing $H$ such that $Orb_G(H)$ has no isolated points.
\item[(b)] $H$ is recursively presented and non-co-Hopfian.
\end{enumerate}
\end{prop}

\begin{proof}
We first assume that $H$ satisfies (a). Every finitely generated subgroup of a finitely presented group is recursively presented, hence so is $H$. To show that $H$ is non-co-Hopfian, we let $\mathcal F$ be a finite generating set of $H$. By our assumption, there exists $g\in G$ such that $gHg^{-1}\ne H$ and $gHg^{-1}\in U(\mathcal F, H)$, where $U(\mathcal F, H)$ is defined by (\ref{U}). In particular, we have $\mathcal F\subseteq gHg^{-1}$. Therefore, $H\lneqq gHg^{-1}\cong H$. Since every finitely generated subgroup of a finitely presented group is recursively presented, we have (b).

Now, suppose that $H$ satisfies (b). Let $\{ H_i\mid i\in \NN\}$, be a set of isomorphic copies of $H$. For each $i\in \NN$, we fix an isomorphism $\alpha_i\colon H\to H_i$. We also fix an injective, non-surjective homomorphism $\beta\colon H\to H$.  Let $K$ denote the fundamental group of the countably infinite star graph of groups shown on Fig. \ref{Fig1} with the following local data. The groups associated to the central  (respectively, every peripheral) vertex is $H$ (respectively, $H_i\cong H$). The group associated to each edge is $H$ and the corresponding embeddings to the vertex groups are $\id \colon H\to H$ for the central vertex and $\alpha_i\circ \beta\colon H\to H_i$ for the corresponding peripheral vertex. Using the standard results about amalgamated free products, it is not difficult to show that $H\le K$ is the accumulation point of subgroups $H_i$ in $Sub(K)$.

Further, let $L$ denote the multiple HNN-extension of $K$ corresponding to the family of isomorphisms $\alpha_i\colon H\to H_i$. Obviously, $Orb_L(H)$ contains all subgroups $H_i$. Hence, $H$ is a limit point of $Orb_L(H)$ (equivalently,  $Orb_L(H)$ has no isolated points). It is easy to see that $L$ can be defined by a recursive presentation whenever $H$ is recursively presented. By the Higman embedding theorem, $L$ embeds into a finitely presented group $G$. The embedding $L\le G$ induces a continuous map $Sub(L)\to Sub(G)$ and, therefore, $Orb_G(H)$ has no isolated points.
\end{proof}

\noindent \textbf{Denis Osin: } Department of Mathematics, Vanderbilt University, Nashville 37240, U.S.A.\\
E-mail: \emph{denis.v.osin@vanderbilt.edu}

\end{document}